\documentclass[english,11pt]{article}
\usepackage[T1]{fontenc}
\usepackage[latin9]{inputenc}
\usepackage{geometry}
\usepackage{array}
\usepackage{longtable}
\usepackage{amsthm}
\usepackage{amsmath}
\usepackage{amssymb}
\usepackage{graphicx}
\usepackage{esint}
\usepackage{color, fancybox, epsfig, epic, eepic}
\makeatletter

\newcommand{\noun}[1]{\textsc{#1}}
\providecommand{\tabularnewline}{\\}

\newenvironment{lyxcode}
{\par\begin{list}{}{
\setlength{\rightmargin}{\leftmargin}
\setlength{\listparindent}{.5pt}
\raggedright
\setlength{\itemsep}{.5pt}
\setlength{\parsep}{.5pt}
\normalfont\ttfamily}%
 \item[]}
{\end{list}}
\definecolor{green}{rgb}{0.2,0.6,0.4}
\definecolor{blue}{rgb}{0.2,0.2,0.7}
\definecolor{red}{rgb}{0.7,0.2,0.2}

\newtheorem{theorem}{Theorem}[section]
\newtheorem{proposition}[theorem]{Proposition}
\newtheorem{lemma}[theorem]{Lemma}

\newtheorem{definition}[theorem]{Definition}
\theoremstyle{plain}

\makeatother

\usepackage{babel}

\title{Efficient algorithms computing distances between Radon measures on
$\mathbb{R}$}
\author{J\c{e}drzej Jab{\l}o{\'n}ski$^{1}$ and Anna Marciniak-Czochra$^{2}$ \\ \\
{$^{1}$ \small \it Institute of Applied Mathematics and Mechanics, } \\
{\it \small University of Warsaw} \\
{\it \small Warszawa 02-097, Poland}\\ \\
{$^{2}$ \small \it{Institute of Applied Mathematics}},\\
{\it \small {Interdisciplinary Center for Scientific Computing (IWR)}},\\ 
{\it \small {and BIOQUANT Center}}\\
{\it \small {University of  Heidelberg}}\\
{\it \small {Im Neuenheimer Feld 294, 69120  Heidelberg, Germany}}}

\begin{document}
\maketitle
\begin{abstract}
In this paper numerical methods of computing distances between two
Radon measures on $\mathbb R$ are discussed. Efficient algorithms for Wasserstein-type
metrics are provided. In particular, we propose a novel algorithm to compute the flat metric (bounded Lipschitz distance) with a computational cost $O(n\log n)$. The flat distance has recently proven to be adequate for the Escalator Boxcar Train (EBT) method for solving transport equations with growth terms. Therefore, finding efficient numerical algorithms to compute the flat distance between two measures is important for finding the residual error and validating empirical convergence of different methods.

\end{abstract}

\section{Introduction}
Recent years witnessed large developments in the kinetic theory methods applied to mathematical physics and more recently also to mathematical biology.  Among important branches of the kinetic theory are optimal transportation problems  and related to them Wasserstein metrics, and Monge-Kantorovich metrics \cite{ambrosio,villani}.  
Partial differential equations in metric spaces are being applied to transportation problems \cite{Kinderlehrer, villani},  gradient flows \cite{ambrosio,Westdickenber} and structured population models \cite{CCGU,GLMC,GMC,two_sex}. 
 
Output of mathematical modeling can often be described as Radon measures. Comparing the results of the models requires then a definition of distance in the space of measures. The desired properties of such metrics depend on the structure of the considered problem. In most of the cases, the topology of total variation is too strong for applications, and weaker metrics had to be introduced, see \cite{GMC} for details. Well known $1$-Wasserstein metric, on the other hand, is only applicable to processes with mass conservation. To cope with growth in the process, various modifications have been proposed, including  flat metric and centralized Wasserstein metric. In the present paper, we additionally introduce a normalized Wasserstein distance. For comparison of different metrics, their interpretation and examples see Table \ref{Table}.  Even though, all those distances can be computed using linear programming (LP), its computational complexity becomes often larger than the complexity of solving the original problem. For example, the equations for which the stability of numerical algorithm is proven in $(W^{1,\infty})^*$ require an efficient algorithm for the flat metric to find the residual error.

Algorithms proposed in this paper are designed to compute efficiently Wasserstein-type distance between two Radon measures in the form of $\sum_{i=1}^{n}a_{i}\delta_{x_{i}}$,
where $\delta$ is the Dirac delta. The algorithms can be also applied for the case of an arbitrary pair of measures, by approximating those measures by the sum of Dirac deltas. 
In the process of finding numerical solutions to partial differential equations, the distance between a discrete and absolutely continuous measures is needed to evaluate the quality of the initial condition approximation, while the distance between two discrete measure is needed to find the residual error.

The paper is organized as follows. In Section \ref{Metrics} we introduce and compare four Wasserstein-type metrics. Additionally, we introduce a two-argument function, based on the Wasserstein distance, which is not a metric, however, provides a good tool to estimate the flat metric from above.  Section \ref{Algorithms} is devoted to numerical algorithms proposed to calculate distances between two measures in respect to the considered metrics. We justify the algorithms and provide respective pseudocodes. The novelty of this paper is the algorithm for flat metric, which has been recently proven to be adequate for the Escalator Boxcar Train (EBT) method for solving transport equations \cite{EBT}. We propose an efficient algorithm which computation cost is $O(n^2)$ and optimize it further to $O(n\log n)$. To judge the efficiency of the algorithms, we compare the times needed to compute the flat distance between two sums of a large number of Dirac deltas randomly distributed over [-1,1]. 

\section{Metrics on the spaces of Radon measures} \label{Metrics}

\subsection{1-Wasserstein distance}

The framework of Wasserstein metric in the spaces of probability measures has proven to be very useful for the analysis of equations given in a conservative form, such as, for example, transport equation \cite{ambrosio}, Fokker-Planck \cite{Fokker-Planck} or nonlinear diffusion  equation \cite{porus_media}.
 It was originally defined using the notion of optimal transportation,  but we focus on its dual representation.

\begin{definition}
$1$-Wasserstein distance between two measures 
$\mu$ and $\nu$ is given by

\begin{equation}\label{Wassersteinm}
W_{1}(\mu,\nu)=\sup\left\{ \left|\int_{\mathbb{R}}f(x)d(\mu-\nu)(x)\right|:f\in C(\mathbb{R},\mathbb{R}),Lip(f)\leqslant1\right\} 
\end{equation}
\end{definition}

\noindent Observe that the integral  $\int_{\mathbb{R}}f(x)d(\mu-\nu)(x)$ is invariant with respect to adding a constant to function $f$.
Consequently, for any $x^*\in\mathbb{R}$ holds

\[
W_{1}(\mu,\nu)=\sup\left\{ \int_{\mathbb{R}}f(x)d(\mu-\nu)(x):\, f\in C(\mathbb{R},\mathbb{R}),f(x^*)=0,Lip(f)\leqslant1\right\}. 
\]

Furthermore, the following two properties hold:

\begin{itemize}
\item $1$-Wasserstein distance is scale-invariant
\[
W_{1}(\lambda\cdot\mu,\lambda\cdot\nu)=\lambda W_{1}(\mu,\nu).
\]
\item $1$-Wasserstein distance is translation-invariant
\[
W_{1}(T_{x}\mu,T_{x}\nu)= W_{1}(\mu,\nu).
\]

\end{itemize}

\subsection{Normalized Wasserstein metric}

If $\mu(\mathbb{R})\neq\nu(\mathbb{R})$, then $W_{1}(\mu,\nu)=\infty$ by the dual representation definition.
It makes this metric useless for processes not preserving the conservation of mass. 

In this section we introduce a simple normalization that leads to a definition of translation-invariant metric suitable for non-local models 

\begin{definition}
We define normalized $1$-Wasserstein distance between two measures $\mu$ and $\nu$ as

\begin{equation}\label{normalizedm}
\widetilde{W}_{1}(\mu,\nu)=\min\left(\left\Vert \mu\right\Vert +\left\Vert \nu\right\Vert ,|\left\Vert \mu\right\Vert -\left\Vert \nu\right\Vert |+W_{1}\left(\frac{\mu}{\left\Vert \mu\right\Vert},\frac{\nu}{\left\Vert \nu\right\Vert}\right)\right),
\end{equation}
where $\left\Vert \mu\right\Vert $ denotes the total variation of the measure $\mu$.
\end{definition}

\begin{lemma}
The distance defined by \eqref{normalizedm} is a metric.
\end{lemma}
\begin{proof}
Let $\mu$, $\nu$ and $\eta$ be Radon measures. Then, it holds
\begin{itemize}
\item[{\it 1.}] $\widetilde{W}_{1}(\mu,\nu) = 0$ if and only if  $\mu=\nu$. 

Indeed,
either $\left\Vert \mu\right\Vert +\left\Vert \nu\right\Vert =0$ or $|\left\Vert \mu\right\Vert -\left\Vert \nu\right\Vert |+W_{1}\left(\frac{\mu}{\left\Vert \mu\right\Vert},\frac{\nu}{\left\Vert \nu\right\Vert}\right)=0$
imply that $\mu=\nu$.
\item[{\it 2.}] $ \widetilde{W}_{1}(\mu,\nu) = \widetilde{W}_{1}(\nu,\mu)$,
\item[{\it 3.}]  Since
\begin{eqnarray*}
\widetilde{W}_{1}(\mu,\nu)+\widetilde{W}_{1}(\nu,\eta) &=& \min\left(\left\Vert \mu\right\Vert +\left\Vert \nu\right\Vert ,|\left\Vert \mu\right\Vert -\left\Vert \nu\right\Vert |+{W}_{1}\left(\frac{\mu}{\left\Vert \mu\right\Vert},\frac{\nu}{\left\Vert \nu\right\Vert}\right)\right)\\&&+\min\left(\left\Vert \eta\right\Vert +\left\Vert \nu\right\Vert ,|\left\Vert \eta\right\Vert -\left\Vert \nu\right\Vert |+{W}_{1}\left(\frac{\eta}{\left\Vert \eta\right\Vert},\frac{\nu}{\left\Vert \nu\right\Vert}\right)\right),
\end{eqnarray*}
to show the triangle inequality, we consider four possibilities
\begin{eqnarray*}
\widetilde{W}_{1}(\mu,\nu)+\widetilde{W}_{1}(\nu,\eta)&=&\left\Vert \mu\right\Vert +\left\Vert \nu\right\Vert +\left\Vert \eta\right\Vert +\left\Vert \nu\right\Vert \geqslant\left\Vert \mu\right\Vert +\left\Vert \eta\right\Vert \geqslant \widetilde{W}_{1}(\mu,\eta),\\
\widetilde{W}_{1}(\mu,\nu)+\widetilde{W}_{1}(\nu,\eta)&=&|\left\Vert \mu\right\Vert -\left\Vert \nu\right\Vert |+W_{1}\left(\frac{\mu}{\left\Vert \mu\right\Vert},\frac{\nu}{\left\Vert \nu\right\Vert}\right)\\
&&+|\left\Vert \eta\right\Vert -\left\Vert \nu\right\Vert |+W_{1}\left(\frac{\eta}{\left\Vert \eta\right\Vert},\frac{\nu}{\left\Vert \nu\right\Vert}\right)\geqslant \widetilde{W}_{1}(\mu,\eta),\\
\widetilde{W}_{1}(\mu,\nu)+\widetilde{W}_{1}(\nu,\eta)&=&\left\Vert \mu\right\Vert +\left\Vert \nu\right\Vert +|\left\Vert \eta\right\Vert -\left\Vert \nu\right\Vert |+W_{1}\left(\frac{\eta}{\left\Vert \eta\right\Vert},\frac{\nu}{\left\Vert \nu\right\Vert}\right)\geqslant\left\Vert \mu\right\Vert +\left\Vert \eta\right\Vert \\ &\geqslant& \widetilde{W}_{1}(\mu,\eta),\\
\widetilde{W}_{1}(\mu,\nu)+\widetilde{W}_{1}(\nu,\eta)&=&\left\Vert \eta\right\Vert +\left\Vert \nu\right\Vert +|\left\Vert \mu\right\Vert -\left\Vert \nu\right\Vert |+W_{1}\left(\frac{\mu}{\left\Vert \mu\right\Vert},\frac{\nu}{\left\Vert \nu\right\Vert}\right)\geqslant\left\Vert \mu\right\Vert +\left\Vert \eta\right\Vert \\ &\geqslant& \widetilde{W}_{1}(\mu,\nu).
\end{eqnarray*}
\end{itemize}
\end{proof}
\noindent This metric can be easily computed numerically using the algorithm for $1$-Wasserstein distance. However, it lacks the scaling property,  which holds for the Wasserstein distance, and which is useful for applications. Nonetheless, the following weaker property holds.

\begin{proposition}
Let ${\mu_k}$ and ${\nu_k}$ be two sequences of Radon measures and $\left\Vert \mu_k \right\Vert \to 0$, $\left\Vert \nu_k \right\Vert \to 0$ then $\widetilde{W}_{1}(\mu_k,\nu_k)\to 0$
\end{proposition}

\subsubsection{Centralized Wasserstein metric}

In this section we present a different modification of  the $1$-Wasserstein distance, which is scale-invariant.

\begin{definition}
We define centralized $1$-Wasserstein distance between two measures $\mu$ and $\nu$ as
\[
\widehat{W}_{1}(\mu,\nu)=\sup\left\{ \left|\int_{\mathbb{R}}f(x)d(\mu-\nu)(x)\right|:f\in C(\mathbb{R},\mathbb{R}),Lip(f)\leqslant1,f(0)\in[-1,1]\right\}. 
\]
\end{definition}

The metric was introduced in \cite{GMC} for analysis of the structured population models in the spaces of non-negative Radon measures, ${\cal M}^+_1({\mathbb R^+})$, on $\mathbb R^+$ with integrable first moment, i.e., $$ {\cal M}_1({\mathbb R^+}):= \Big\{ \mu \in {\cal M^+}({\mathbb R^+}) \; \big| \;
  \int_{\mathbb R^+} \: |x| \; d \mu < \infty \Big\}.
$$

This metric satisfies the scaling property, but on the contrary to Wasserstein metric,  it is not invariant with respect to translations. It is therefore only useful for modeling specific phenomena, for example  structured population dynamics where mass generation can only occur in one specific point of space.

\subsection{Flat metric}

Another solution of the problem of comparing two measures of different
masses is the flat metric which is defined as follows
\begin{definition}
Flat distance between two measures 
$\mu$ and $\nu$ is given by
\[
F(\mu,\nu)=\sup\left\{ \left|\int_{\mathbb{R}}f(x)d(\mu-\nu)(x)\right|:f\in B_{C(\mathbb{R},\mathbb{R})}(0,1),Lip(f)\leqslant1\right\}. 
\]
\end{definition}
\noindent The flat metric, known also as a bounded Lipschitz distance \cite{Neunzert}, corresponds to the dual norm of $W^{1, \infty}(\mathbb{R})$, since the test functions used in the above definition are dense in $W^{1,\infty}(\mathbb R)$.
This metric satisfies the scaling property and it is translation invariant. It has proven to be useful in analysis of structured population models and in particularly,  Lipschitz dependence of solutions on the model parameters and initial data \cite{GLMC,CCGU}.  The flat metric has been recently used for the proof of convergence and stability of EBT, which is a numerical algorithm based on particle method, for the transport equation with growth terms \cite{EBT}. Consequently, an implementation of the EBT method requires an algorithm to compute the flat distance between two measures. 

\subsection{Upper bound for flat metric}

Computing of flat metric exactly is more costly than computing of  $1$-Wasserstein metric ($O(n\log n)$ vs $O(n)$, see Section \ref{SectionFlatAlg}). Moreover, often in applications it is sufficient to calculate upper bound of the distance, for example when computing residual error.  We propose the following function, which requires only linear computing time.

\begin{eqnarray*}
\widehat F(\mu, \nu) = |\left\Vert \mu\right\Vert -\left\Vert \nu\right\Vert |+\begin{cases}
W_{1}\left(\mu,\nu\frac{\left\Vert \mu\right\Vert}{\left\Vert \nu\right\Vert}\right) & \mbox{if }\left\Vert \mu\right\Vert< \left\Vert \nu\right\Vert\\
W_{1}\left(\mu\frac{\left\Vert \nu\right\Vert}{\left\Vert \mu\right\Vert},\nu\right) & \mbox{if }\left\Vert \mu\right\Vert \geq \left\Vert \nu\right\Vert \end{cases}
\end{eqnarray*}

\begin{lemma}
Let $\mu$, $\nu$ be Radon measures on $\mathbb R$. Then, the following estimate holds
\[
F(\mu,\nu) \leq \widehat F(\mu, \nu)
\]
\end{lemma}
 \begin{proof}
 Assume that $\left\Vert \mu\right\Vert< \left\Vert \nu\right\Vert$. Then, using the estimate of the flat metric by $1$-Wasserstein distance for the measures with equal masses provides
 \begin{eqnarray*}
 F(\mu, \nu) &\leq& F(\mu, \nu\cdot \frac{\left\Vert \mu\right\Vert}{\left\Vert \nu\right\Vert}) + F(\nu\cdot \frac{\left\Vert \mu\right\Vert}{\left\Vert \nu\right\Vert}, \nu)\\
 &\leq& W_{1}\left(\mu,\nu\cdot \frac{\left\Vert \mu\right\Vert}{\left\Vert \nu\right\Vert}\right) + F(\nu\cdot \frac{\left\Vert \mu\right\Vert}{\left\Vert \nu\right\Vert}, \nu) \\
 &=& W_{1}\left(\mu,\nu\cdot \frac{\left\Vert \mu\right\Vert}{\left\Vert \nu\right\Vert}\right)  + \sup\left\{ \left|\int_{\mathbb{R}}f(x)(  \frac{\left\Vert \mu\right\Vert}{\left\Vert \nu\right\Vert}-1)  d \nu (x)\right|:f\in B_{C(\mathbb{R},\mathbb{R})}(0,1),Lip(f)\leqslant1\right\}\\
  &=& W_{1}\left(\mu,\nu\cdot \frac{\left\Vert \mu\right\Vert}{\left\Vert \nu\right\Vert}\right)  +  \left| \frac{\left\Vert \mu\right\Vert}{\left\Vert \nu\right\Vert}-1\right| \|\nu\|.
 \end{eqnarray*}
 The last equality results from the representation of Radon distance and positivity of  measure $\nu$.
 \end{proof}

The upper bound function is more useful if it can be estimated from above by $c\cdot F(\mu,\nu)$ for some constant $c$. In a general case, however, such constant does not exist, since by taking $\mu=\delta_0$, $\nu=\delta_x$ and passing $x\to\infty$ we obtain
\begin{eqnarray*}
 \widehat F(\mu, \nu)& \to& \infty,\\
 F(\mu,\nu) &\to&  2.
\end{eqnarray*}
Nevertheless, the desired estimate can be shown on a bounded set.

\begin{lemma}
Let $\mu$, $\nu$ be non-negative Radon measures on a compact set $K$, then the following estimate holds
\[
\widehat F(\mu, \nu) \leq c_K F(\mu, \nu)
\]
\end{lemma}
\begin{proof}
Let $c_K=max(1,\frac{1}{2} diam(K))$ and $x_0$ be such point that $|k-x_0|\leq c_K$ for any $k\in K$. Then, assuming $\left\Vert \mu\right\Vert< \left\Vert \nu\right\Vert$, we obtain

\begin{eqnarray*}
\widehat F(\mu, \nu) &=& |\left\Vert \mu\right\Vert -\left\Vert \nu\right\Vert | + W_{1}\left(\mu,\nu\frac{\left\Vert \mu\right\Vert}{\left\Vert \nu\right\Vert}\right) \\
&=& |\int d(\mu-\nu)| + W_{1}\left(\mu,\nu\frac{\left\Vert \mu\right\Vert}{\left\Vert \nu\right\Vert}\right) \\
&\leq& F(\mu,\nu) + W_{1}\left(\mu,\nu\frac{\left\Vert \mu\right\Vert}{\left\Vert \nu\right\Vert}\right) \\
&\leq& F(\mu,\nu) + W_{1}\left(\mu,\nu\right)+W_{1}\left(\nu, \nu\frac{\left\Vert \mu\right\Vert}{\left\Vert \nu\right\Vert}\right) \\
&\leq& 2\cdot F(\mu,\nu) + W_{1}\left(\mu,\nu\right) \\
&=& 2\cdot F(\mu,\nu) + \sup\left\{ \left|\int_{\mathbb{R}}f(x)d(\mu-\nu)(x)\right|:f\in C(\mathbb{R},\mathbb{R}),Lip(f)\leqslant1\right\} \\
&=& 2\cdot F(\mu,\nu) + c_K \sup\left\{ \left|\int_{\mathbb{R}}f(x)d(\mu-\nu)(x)\right|:f\in C(\mathbb{R},\mathbb{R}),Lip(f)\leqslant\frac{1}{c_K}\right\} \\
&\leq& (2+c_K) F(\mu,\nu) \\
\end{eqnarray*}
\end{proof}
 
\subsection{Comparison of the presented metrics}

In the following Table we compare the introduced metrics and provide examples and interpretations in terms of optimal transport.

\begin{longtable}{|>{\centering}p{1.5cm}|>{\centering}p{2.7cm}|>{\centering}p{1.4cm}|>{\centering}p{1.8cm}|>{\centering}p{5cm}|>{\centering}p{1.9cm}|}
\hline 
Metric  & Example: $d(2\delta_{x},3\delta_{y})$  & Scale-invariance & Translation-invariance & Intuition of $d(\mu,\nu)$  & Compute complexity\tabularnewline
\hline 
\hspace{-0.1cm}Wasserstein  & $\infty$  & YES  & YES  & The cost of optimal transportation of distribution $\mu$ to a state
given by $\nu$, assuming that moving mass $m$ by $x$ requires $mx$
energy.  & $O(n)$\tabularnewline
\hline 
\hspace{-0.1cm}Wasserstein normalized  & $\min(2+3,$ $(3-2)+|x-y|)$ & \multicolumn{1}{c|}{weak} & \multicolumn{1}{c|}{YES } & Minimum of the cost of annihilating $\mu$ and generating $\nu$; and
of the cost of generating the difference in mass between $\mu$ and $\nu$
and transporting $\frac{\mu}{\left\Vert \nu\right\Vert }$ to $\frac{\nu}{\left\Vert \nu\right\Vert }$,
assuming that generating/annihilating mass at any position requires the cost equal to the mass.  & $O(n)$\tabularnewline
\hline 
\hspace{-0.1cm}Wasserstein centralized  & $2|x-y|+|y|$  & YES  & NO  & The cost of generating the difference in mass in point $0$ in space
added to the cost of transporting $\mu+\left(\left\Vert \nu\right\Vert -\left\Vert \mu\right\Vert \right)\delta_{0}$
to $\nu$.  & $O(n)$\tabularnewline
\hline 
Flat  & $1+2\min\left(2,|x-y|\right)$  & YES  & YES  & The cost of optimal transporting AND/OR generating AND/OR annihilating
mass to form $\nu$ from $\mu$.  & $O(n\log n)$\tabularnewline
\hline 
Radon  & $2+3$  & YES  & YES  & The cost of generating AND/OR annihilating mass to form $\nu$ from
$\mu$  & $O(n)$\tabularnewline
\hline 

\caption{Comparison of different metrics.}\label{Table}
\end{longtable}

\section{Approximations and computational algorithms}\label{Algorithms}

In the remainder of this paper, we present algorithms for computing and approximating of the introduced metrics.

\subsection{Computing of $1$-Wasserstein distance}

We start with introducing the algorithm for computation of $1$-Wasserstein distance. The algorithm is well-known and straightforward. Nevertheless, we present it here, since the specific approach used in this section will be also applied later in the proof of correctness of more involved algorithms for other distances. 

\subsubsection{Reduction to the case of discrete measures}

\begin{lemma}

Let $\mu$ be an arbitrary Radon measure supported on $[0,1]$. For every $\varepsilon >0$ there exist a discrete measure $\mu_{\varepsilon}$ such that

\[
W_{1}(\mu, \mu_{\varepsilon}) < \varepsilon
\]
\end{lemma}
\begin{proof}
Define
\begin{eqnarray*}
{\mu_{\frac{\mu[0,1]}{n}}}&=&\sum_{i=1}^{n}\delta_{i/n}\mu\left[\frac{i-1}{n},\frac{i}{n}\right).\\
\end{eqnarray*}
Then, we estimate
\[
W_{1}(\mu,\mu_{\frac{\mu[0,1]}{n}}) \leq \sum_{i=1}^{n}\frac{1}{n} \mu\left[\frac{i-1}{n},\frac{i}{n}\right)
\leq \frac{1}{n} \mu[0,1],
\]
and the right-hand side tends to $0$ with $n\to \infty$.
\end{proof}

Consequently, $1$-Wasserstein distance between an arbitrary pair of finite Radon measures can be estimated by the distances between their discrete approximations. From this point on, we assume that
\[
\mu-\nu=\sum_{k=1}^{n}a_{k}\delta_{x_{k}}.
\]

\subsubsection{The algorithm}
The formula for $1$-Wasserstein distance reads

\[
W_{1}(\mu,\nu)=\sup\left\{ \sum_{k=1}^{n}a_{k}f(x_{k}):f\in C(\mathbb{R},\mathbb{R}),f(x_{n})=0,Lip(f)\leqslant1\right\}. 
\]
Regularity conditions can be represented as linear programming bounds. Hence, computing of $W_{1}(\mu,\nu)$ is equivalent to finding maximum of
\[
\left|\sum_{k=1}^{n}a_{k}f_{k}\right|
\]
with the following restrictions

\begin{eqnarray*}
f_{n} & = & 0,\\
\left|f_{k}-f_{k-1}\right| & \leqslant & \left|x_{k}-x_{k-1}\right|.
\end{eqnarray*}
Although this problem can be solved by linear programming, a more efficient
algorithm can be found.

Define
\[
W_{1}^{m}(x)=\sup\left\{ \sum_{k=1}^{m}a_{k}f_{k}:f_{m}=x,\left|f_{k}-f_{k-1}\right|\leqslant\left|x_{k}-x_{k-1}\right|\right\}. 
\]
Then, obviously $W_{1}(\mu,\nu)=W_{1}^{n}(0)$. Denote $d_{k}=x_{k+1}-x_{k}$,
and observe that
\[
W_{1}^{1}(x)=a_{1}x,
\]
\begin{eqnarray*}
W_{1}^{2}(x) & = & a_{2}x+\sup_{f_{1}\in[x-d_{1},x+d_{1}]}W_{1}^{1}(x)=a_{2}x+a_{1}x+a_{1}\cdot sgn(a_{1})d_{1}=\\
 & = & \left(a_{1}+a_{2}\right)x+|a_{1}|d_{1}.
\end{eqnarray*}
It can be shown by induction that
\[
W_{1}^{n}(x)=\left(\sum_{i=1}^{n}a_{i}\right)x+\sum_{i=1}^{n-1}d_{i}\left|\sum_{j=1}^{i}a_{j}\right|.
\]
Notice that the value $a_{n}$ is not used in this formula. It is,
however, involved indirectly, as $\sum_{i=1}^{n}a_{i}=0$.

\subsubsection{Pseudocode}
\begin{lyxcode}
\textbf{Input:~}Non-decreasing~table~of~positions\textbf{~$x\in[0,1]^{n}$,~}table~of~masses\textbf{~$a\in\mathbb{R}^{n}$}

\noun{1-Wasserstein-Distance}~($x\in[0,1]^{n}$,~$a\in\mathbb{R}^{n}$)
\begin{lyxcode}
$distance$~$\leftarrow$~$0$

$partialSum$~$\leftarrow$~$0$

\textbf{for}~$idx$$\leftarrow$~$1$~\textbf{to}~$n-1$~\textbf{do}
\begin{lyxcode}
$partialSum$~$\leftarrow$~$partialSum+a_{idx}$

$distance$~$\leftarrow$~$distance+(x_{idx+1}-x_{idx})\cdot\left|partialSum\right|$
\end{lyxcode}
\textbf{return}~$distance$
\end{lyxcode}
\end{lyxcode}

\subsubsection{Complexity of the algorithm}

It is clear from the pseudocode that the computational complexity
of the algorithm is $\Theta(n)$, while memory complexity is $\Theta(1)$.

\subsection{Computing of  the centralized Wasserstein distance}

\subsubsection{Reduction to the case of discrete measures}

Similarly to the case of $1$-Wasserstein distance the following lemma holds.

\begin{lemma}

Let $\mu$ be an arbitrary Radon measure supported on $[0,K]$. For every $\varepsilon >0$, there exists a discrete measure $\mu_{\varepsilon}$ such that

\[
{\widehat W}_{1}(\mu, \mu_{\varepsilon}) < \varepsilon.
\]
\end{lemma}

The proof follows from the fact that $\left\Vert \mu\right\Vert=\left\Vert \nu\right\Vert$ implies ${\widehat W}_{1}(\mu,\nu)=W_{1}(\mu,\nu)$.

\subsubsection{The Algorithm}

We assume
\[
\mu-\nu=\sum_{i=1}^{m}a_{i}\delta_{x_{i}}+a_{m+1}\delta_{0}+\sum_{i=m+2}^{n}a_{i}\delta_{x_{i}}.
\]
Define
\begin{eqnarray*}
\underline{W}_{1}^{j}(x) & = & \sup\left\{ \sum_{k=1}^{j}a_{k}f_{k}:f_{j}=x,\left|f_{k}-f_{k-1}\right|\leqslant\left|x_{k}-x_{k-1}\right|\right\}, \\
\overline{W}_{1}^{j}(x) & = & \sup\left\{ \sum_{k=j}^{n}a_{k}f_{k}:f_{j}=x,\left|f_{k}-f_{k-1}\right|\leqslant\left|x_{k}-x_{k-1}\right|\right\}. 
\end{eqnarray*}
As already proven
\begin{eqnarray*}
\underline{W}_{1}^{m+1}(x) & = & \left(\sum_{i=1}^{m+1}a_{i}\right)x+\sum_{k=1}^{m}d_{k}\left|\sum_{i=1}^{k}a_{i}\right|,\\
\overline{W}_{1}^{m+1}(x) & = & \left(\sum_{i=m+1}^{n}a_{i}\right)x+\sum_{k=1}^{n-(m+1)}d_{n-k}\left|\sum_{i=n+1-k}^{n}a_{i}\right|.
\end{eqnarray*}
From LP representation of the metric
\[
{\widehat W}_{1}(\mu,\nu)=\sup\left\{ \sum_{k=1}^{n}a_{k}f_{k}:-1\leqslant f_{m+1}\leqslant1,\left|f_{k}-f_{k-1}\right|\leqslant\left|x_{k}-x_{k-1}\right|\right\}, 
\]
it can be deduced that
\[
{\widehat W}_{1}(\mu,\nu)=\sup_{x\in[-1,1]}\left(\underline{W}_{1}^{m+1}(x)+\overline{W}_{1}^{m+1}(x)-a_{m+1}x\right),
\]
so the distance is given by the formula
\[
{\widehat W}_{1}(\mu,\nu)=\sum_{k=1}^{m}d_{k}\left|\sum_{i=1}^{k}a_{i}\right|+\sum_{k=1}^{n-(m+1)}d_{n-k}\left|\sum_{i=n+1-k}^{n}a_{i}\right|+\left|\sum_{i=1}^{n}a_{i}\right|.
\]

\subsubsection{Pseudocode}
\begin{lyxcode}
\textbf{Input:~}Non-decreasing~table~of~positions\textbf{~$x\in[0,1]^{k}\times\{0\}\times[0,1]^{n-k-1}$,~}table~of~masses\textbf{~$a\in\mathbb{R}^{n}$}

\noun{1-Wasserstein-Centralized-Distance}
\begin{lyxcode}
$distance$~$\leftarrow$~$0$

$(partialSumFront,\, partialSumBack)$~$\leftarrow$~$(0,\,0)$

$(idxFront,\, idxBack)\leftarrow(1,n)$

\textbf{while~$x_{idxFront}<0$}~\textbf{do}
\begin{lyxcode}
$partialSumFront$~$\leftarrow$~$partialSumFront+a_{idxFront}$

$distance$~$\leftarrow$~$distance+(x_{idxFront+1}-x_{idxFront})\cdot\left|partialSumFront\right|$

$idxFront\leftarrow idxFront+1$
\end{lyxcode}
\textbf{while~$x_{idxBack}>0$}~\textbf{do}
\begin{lyxcode}
$partialSumBack$~$\leftarrow$~$partialSumBack+a_{idxEnd}$

$distance$~$\leftarrow$~$distance+(x_{idxBack}-x_{idxBack-1})\cdot\left|partialSumBack\right|$

$idxBack\leftarrow idxBack-1$
\end{lyxcode}
\textbf{for~$idx\leftarrow idxFront$}~\textbf{to~$idxBack$~do}
\begin{lyxcode}
$partialSumFront\leftarrow partialSumFront+a_{idx}$
\end{lyxcode}
\textbf{return}~$distance+\left|partialSumFront+partialSumBack\right|$
\end{lyxcode}
\end{lyxcode}

\subsubsection{Complexity of the algorithm}

Each iteration of each loop takes a constant time. The total number
of iterations in all three loops is equal to $k+1+\left(n-k-1\right)$.
Computational complexity of this algorithm is therefore $\Theta(n)$,
while the memory complexity is $\Theta(1)$.

\label{SectionFlatAlg}

\subsection{Computing of flat distance}

The algorithm for flat distance requires storing the shape of functions analogous to $W_{1}^{m}$ as they get more complicated when $m$ increases. In Section \ref{flat_alg} we provide a recursive formula for the sequence of these functions. The pseudocode in Section \ref{flat_pseudo} implements the algorithm using an abstract data structure to store previously defined functions. The computational complexity depends on the choice of this structure. In further sections we provide two solutions that require $O(n^2)$ and $O(n\log n)$ operations.

\subsubsection{Reduction to the case of discrete measures}

Similarly to the previous cases the following lemma holds.

\begin{lemma}

Let $\mu$ be an arbitrary Radon measure supported on $[0,K]$. For every $\varepsilon >0$ there exist a discrete measure $\mu_{\varepsilon}$ such that

\[
F(\mu, \mu_{\varepsilon}) < \varepsilon
\]
\end{lemma}

The proof follows from the fact that $\left\Vert \mu\right\Vert=\left\Vert \nu\right\Vert$ implies $F(\mu,\nu) \leq W_{1}(\mu,\nu)$.

\subsubsection{The Algorithm} \label{flat_alg}

We assume

\[
\mu-\nu=\sum_{i=1}^{n}a_{i}\delta_{x_{i}}.
\]
Computing of $F(\mu,\nu)$ is equivalent to finding maximum of

\[
\left|\sum_{k=1}^{n}a_{k}f_{k}\right|
\]
with the following restrictions

\begin{eqnarray*}
\left|f_{k}\right| & \leqslant & 1,\\
\left|f_{k}-f_{k-1}\right| & \leqslant & \left|x_{k}-x_{k-1}\right|.
\end{eqnarray*}
Define
\[
F^{m}(x)=\sup\left\{ \left|\sum_{k=1}^{m}a_{k}f_{k}\right|:f_{m}=x,\left|f_{k}-f_{k-1}\right|\leqslant\left|x_{k}-x_{k-1}\right|,\left|f_{k}\right|\leqslant1\right\}. 
\]
Obviously, it holds by definition $$F(\mu,\nu)=\sup_{x\in[-1,1]}F^{n}(x).$$
Observe that
\begin{eqnarray*}
F^{1}(x)&=&a_{1}x,\\
F^{2}(x) & = & a_{2}x+\sup_{f_{1}\in[x-d_{1},x+d_{1}]\cap[-1,1]}F^{1}(x) = a_{2}x+\min(|a_{1}|,a_{1}x+|a_{1}|d_{1}),\\\\\\
F^{m}(x) & = & a_{m}x+\sup_{f_{m-1}\in[x-d_{m-1},x+d_{m-1}]\cap[-1,1]}F^{m-1}(x).
\end{eqnarray*}
Computing of $F^{m}(x)$ based on $F^{m-1}$ is more complex than in previous cases, as $F^{m-1}$ is not necessarily monotonic. 
\begin{lemma}
 Function $F^{m}$ is concave for each $m$.
\end{lemma}
\begin{proof}\label{concave}
To prove the lemma we will use induction with respect to $m$. $F^{1}(x)$ is given as $a_{1}x$, so it is indeed concave. 
Assume $F^{m}$ is concave. Define 

$$F_{max}^{n,d}(x)=\sup_{[x-d,x+d]\cap[-1,1]}F^{n}(x).$$ 
Choose $x,y\in[-1,1]$. Then, there exist $x'\in B(x,d)\cap[-1,1],\, y'\in B(y,d)\cap[-1,1]$ such that
\[
\alpha F_{max}^{m,d}(x)+(1-\alpha)F_{max}^{m,d}(y)=\alpha F^{m}(x')+(1-\alpha)F^{m}(y').
\]
Because $F^{m}$ is concave, it holds
\[
\alpha F^{m}(x')+(1-\alpha)F^{m}(y')\leqslant F^{m}\left(\alpha x'+(1-\alpha)y'\right)\leqslant F_{max}^{m,d}(\alpha x+(1-\alpha)y)
\]
The last inequality follows from $\alpha x'+(1-\alpha)y'\in B(\alpha x+(1-\alpha)y,d)$.
It is now proven that $F^{m+1}(x)$ is convex, as it is a sum of a linear function and a convex function $F_{max}^{m,d}(y)$.
\end{proof}
\begin{lemma}
For each $m$ the function $F^{m}$ is piecewise linear on $m$ intervals and it holds for some point $x_m$
\[
F^{m}= a_{m}x+\begin{cases}
F^{m-1}(x+d_{m-1}) & \mbox{on \ensuremath{[-1,x_m-d_{m-1}]}}\\
F^{m-1}(x_m) & \mbox{on \ensuremath{[x_m-d_{m-1},x_m+d_{m-1}]}}\\
F^{m-1}(x-d_{m-1}) & \mbox{on \ensuremath{[x_m+d_{m-1},1]}}
\end{cases}
\]
\end{lemma}
\begin{proof}
$F^{1}$ is a linear function, so it can be described by its values in $\pm1$. 
Assume that $F^{m}$ can be described by at most $m+1$ points and is linear between those points.
As $F^{m}$ is concave, there exists a point $x_{m}\in[-1,1]$ such that $F^{m}(x)\leq F^{m}(x_{m})$ for every $x$. The maximum of $F^{m}$ on an interval whose both ends are smaller than $x_{m}$ is taken on its right end. Similarly, if both ends of the intervals are larger than $x_{m}$, the maximum is taken on its left end. Finally, if the interval contains $x_{m}$, the maximum is in $x_{m}$. These considerations prove the formula for $F^{m+1}$.
It follows that $F_{m+1}$ is piecewise linear and it can be described by as many points as $\left.F^{n}\right|_{[d,1-d]}$ plus $1$.
\end{proof}

\subsubsection{Pseudocode}\label{flat_pseudo}
In the following code a data structure called 'funcDescription' being a set of pairs will be used to describe $F^{idx}$. Its has the following interpretation: 
\begin{enumerate}
\item $F^{idx}(-1)=leftValue$
\item if $(v,p)\in funcDescription$ then $F^{idx}(x)'=p$ for all $x$ larger than $v$ and smaller than the next value in the structure.
\end{enumerate}

The last instruction in the main loop of the pseudocode makes it inefficient to implement 'funcDescription' as a simple BST tree.

\begin{lyxcode}
\textbf{Input:~}Non-decreasing~table~of~positions\textbf{~$x\in[0,1]^{n}$,~}table~of~masses\textbf{~$a\in\mathbb{R}^{n}$}

\noun{Flat-Distance}
\begin{lyxcode}
$leftValue\leftarrow0$

$funcDescription\leftarrow\left\{ (-1,0),\,(1,-\infty)\right\} $

\textbf{for~$idx\leftarrow1$}~\textbf{to~$n$~do}
\begin{lyxcode}
$d\leftarrow x_{idx}-x_{idx-1}$

$funcLeft\leftarrow\left\{ (v-d,p):(v,p)\in funcDescription\wedge p>0\right\} $

$funcRight\leftarrow\left\{ (v+d,p):(v,p)\in funcDescription\wedge p<0\right\} $

$v_{m}\leftarrow\min\left\{ v:(v,p)\in funcRight\right\} $

$funcDescription\leftarrow funcLeft\cup\left\{ \left(v_{m}-2d,0\right)\right\} \cup funcRight$

$leftValue\leftarrow leftValue+\sum_{(v,p)\in funcDescription,\, v<1}\left(\min(\min\left\{ v':(v',\_)\in funcDescription\wedge v'>v\right\} ,-1)-v\right)\cdot p$

$(v_{min},p_{min})\leftarrow\max\left\{ (v,p):(v,p)\in funcDescription[i]\wedge v\leqslant-1\right\} $

$(v_{max},p_{max})\leftarrow\max\left\{ (v,p):(v,p)\in funcDescription[i]\wedge v\in[0,1]\right\} $

$funcDescription\leftarrow\left(funcDescription\cap\left\{ (v,p):v\in[0,1]\right\} \right)\cup\left\{ (\max(v_{min},-1),p_{min})\right\} \cup\left\{ (1,-\infty)\right\} $

$funcDescription\leftarrow\left\{ (x,p+a_{n}):(x,p)\in funcDescription\right\} $
\end{lyxcode}
\textbf{return}~$leftValue+\sum_{(v,p)\in funcDescription,\, p>0}\left(\min\left\{ v':(v',\_)\in funcDescription\wedge v'>v\right\} -v\right)\cdot p$
\end{lyxcode}
\end{lyxcode}

\subsubsection{Complexity $O(n^{2})$ of the algorithm}

As mentioned before, the complexity of this algorithm depends on the implementation of $funcDescription$
data structure. 

The simplest implementation of $funcDescription$ uses an array of
pairs $(v,p)$ sorted by $v$ in ascending order and by $p$ in the
reverse order in the same time. This is possible based on Lemma \ref{concave}.

The first block of instructions can be performed in $\Theta(\#funcDescription)$
by simply shifting all elements such that $p<0$ to the right, and
modifying $v$ by iterating over all elements of $funcDescription$.
The next block (computing of $leftValue$) can be computed with the same
complexity, as $\min\left\{ v':(v',\_)\in funcDescription\wedge v'>v\right\} $
is simply the next element after $v$ in the ordered array. Finally,
every instruction in the last block can be performed in $\Theta(\#funcDescription)$
by iterating over all its elements.

In each iteration of the main loop at most $1$ element is added to
$funcDescription$. Therefore the computational complexity of the
algorithm is $O(n^{2})$ while the memory complexity is $O(n)$.

\subsubsection{Complexity $O(n\log n)$ of the algorithm}

The previous result can be improved to $O(n\log n)$ by using balanced
binary search trees data structure.

In this implementation $funcDescription$ is represented by global
variables $p_{modifier}$ and a BST of values $(\Delta v,p)$ where
$p$ is the key. An entry $(v,p)$ in this structure is represented as
\[
\left(\sum_{(\Delta v',p')\in funcDescription\wedge p'\geqslant p}\Delta v',p\right)
\]
so obtaining an element of $funcDescription$ may take linear time. 

The advantages of this structure can be easily seen when analyzing
the first block of the code. The division of $funcDescription$ by
the value of $p$ (at first $0$) can be achieved in $O(\log n)$.
Shifting all elements of those subsets can then be done in a constant
time by modifying first elements of those sets. Adding the
extra node also requires $O(\log n)$ operations.

Setting $leftValue$ may require linear time, but all (apart from
one) visited nodes will be removed in the third block. In all iterations
of the main loop this instruction takes, therefore, $O(n)$.

Removing nodes with the first coordinate $\leqslant-1$ is obviously
done in $O(n)$ in total. Identifying nodes with the first coordinate
$\geqslant1$ might seem problematic. It is, however, known that for
the least $p$ the value of $v$ is equal to $1+d$. Relevant nodes can be,
therefore, removed from the back in $O(n)$.
Adding $a_{n}$ to the second coordinate of each node is done by 
adding it to global variable $p_{modifier}$.

All iterations of the main loop require $O(n\log n+n)$ operations.
The memory complexity is also $O(n\log n)$.

\subsubsection{Performance of the algorithm for the flat distance}

Performance of the algorithm depends on the choice of $funcDescription$
data structure. Theoretic bounds for computational complexity are,
however, not sufficient to argue about performance of these two options.
The first reason is that the each operation in $O(n^{2})$ algorithm
is much faster than in $O(n\log n)$ in terms of number of instructions.
Secondly, hardware architectures provide solutions in which iterating
over large tables is accelerated. Finally, the algorithm does reach
its theoretical bound only if many points concentrate on a small interval. A gap of size $2$ between two points completely cleans $funcDescription$
data structure. All that is shown in numerical tests. To measure performance we have used a single core of AMD Athlon II X4 605e processor clocked at 2.3Ghz with 8GB of memory. The results are presented in Figs.\ref{Fig1}-\ref{Fig2}.

\newpage

\begin{figure}[htbp]
\includegraphics[width=5.7in]{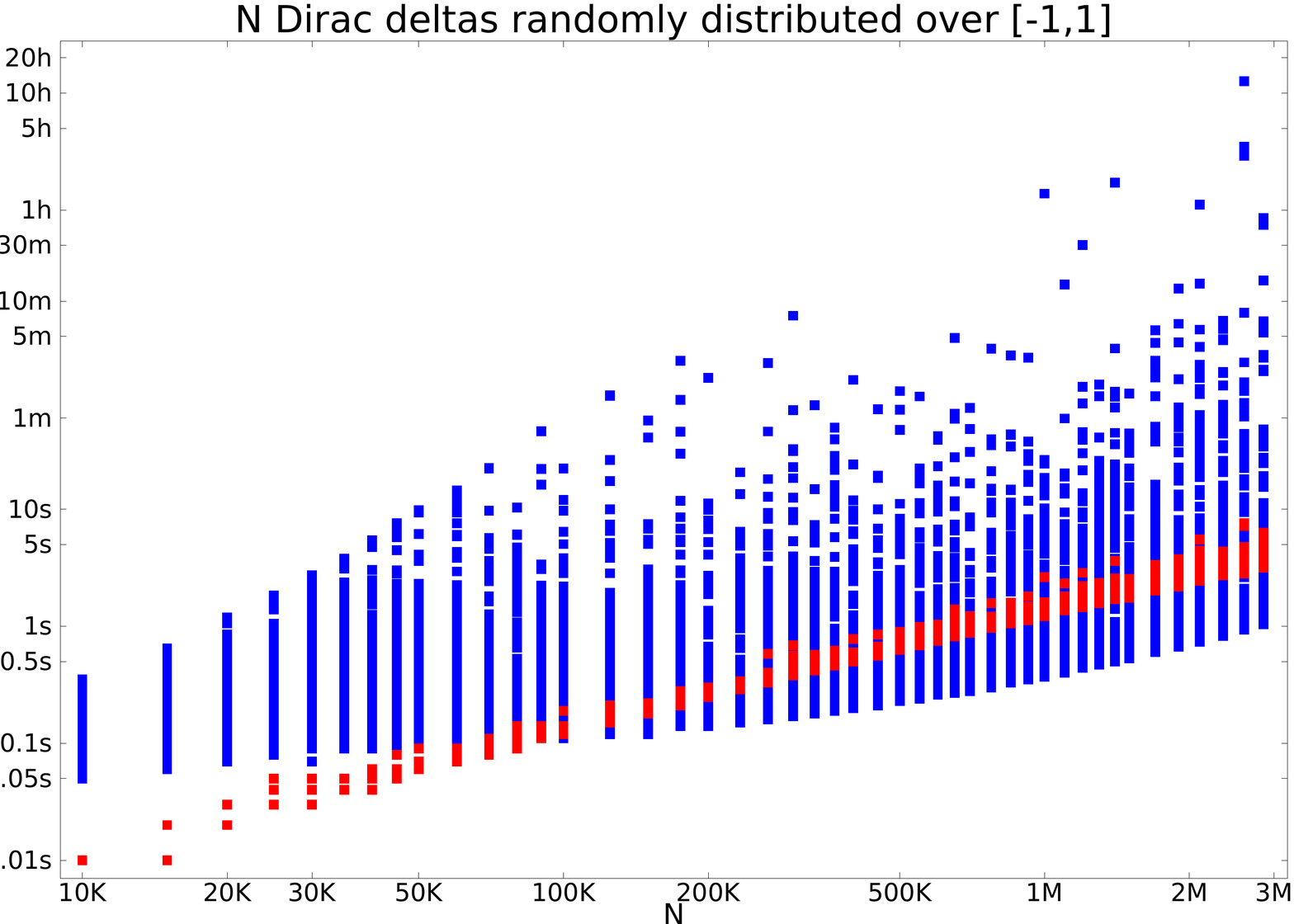}
\caption{Comparison of the performance of the two proposed algorithms for the flat distance for $N$ Dirac deltas randomly distributed over $[-1,1]$. The plot shows how the time of computation depends on $N$. For each input size 100 independent tests were executed to demonstrate how sensitive the algorithms are to input data distribution. Results of $O(nlogn)$ algorithm are depicted as red dots, and results of $O(n^2)$ algorithm as blue dots.}\label{Fig1}
  \end{figure}
  \begin{figure}[htbp]
\includegraphics[width=5.7in]{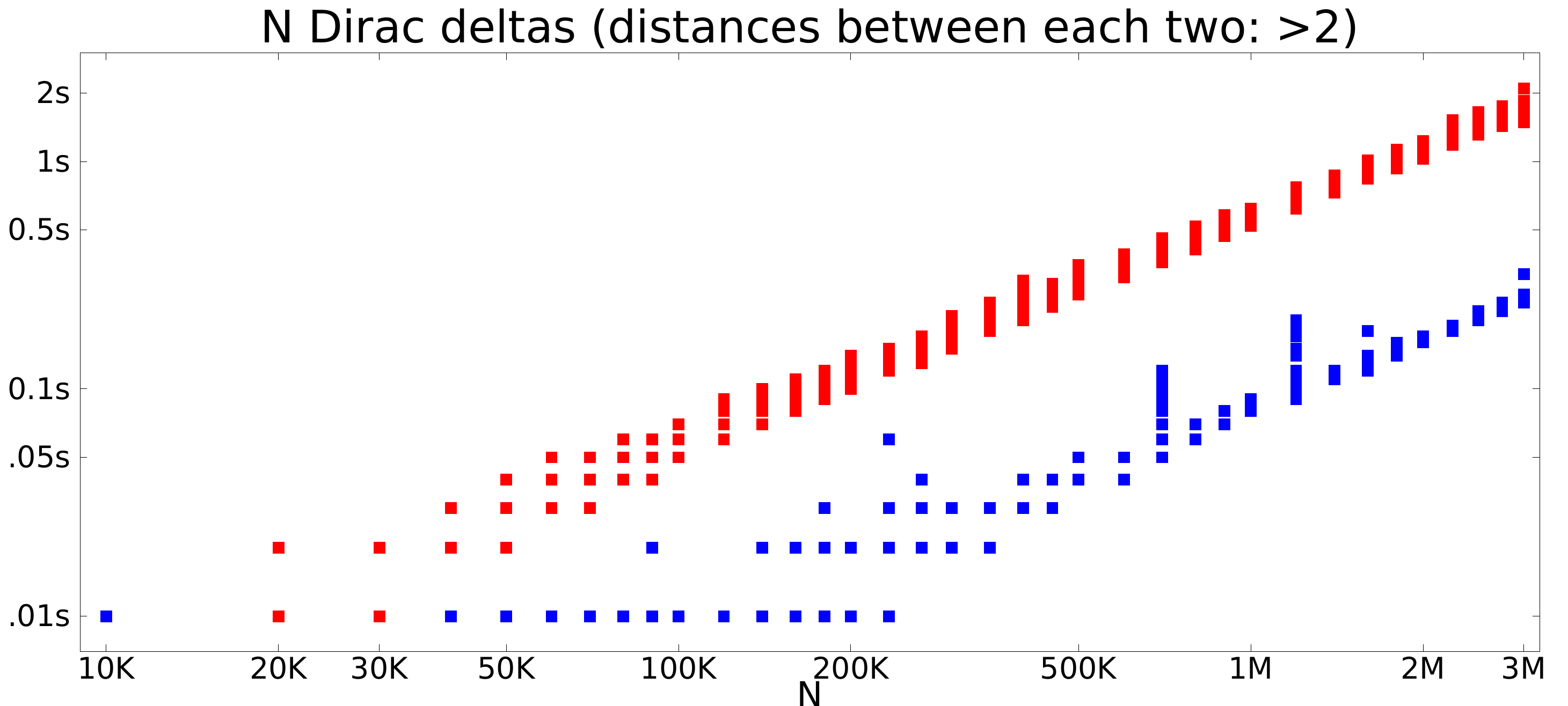}
\caption{Comparison of the performance of the two proposed algorithms for the flat distance for $N$ Dirac deltas with a distributed over a large domain, i.e. distance between each two masses is larger than $2$. In this case both algorithms are in fact linear, as the funcDescription structure has at most two elements. The plot demonstrates the overhead of using BST structures. Results of $O(nlogn)$ algorithm are depicted as red dots, and results of $O(n^2)$ algorithm as blue dots.}\label{Fig2}
  \end{figure}

\newpage
\noindent \textbf{Acknowledgments} \\
JJ was supported by the International Ph.D. Projects Program of Foundation
for Polish Science operated within the Innovative Economy Operational Program
2009-2015 funded by European Regional Development Fund (Ph.D. Program:
Mathematical Methods in Natural Sciences). AM-C was supported by the ERC Starting Grant "Biostruct" No. 210680 and the Emmy Noether Program of the German Research Council (DFG).

\end{document}